\documentclass[12pt]{amsart}

\usepackage{graphicx,diagram}
\usepackage{comment,color}
\usepackage{amssymb}
\usepackage{amsthm}
\usepackage{amsmath}
\usepackage[all]{xy}
\usepackage{amscd}
\usepackage[shortlabels]{enumitem}

\newtheorem{thm}{Theorem}[section]
\newtheorem{cor}[thm]{Corollary}
\newtheorem{lem}[thm]{Lemma}

\newtheorem{eq}[thm]{Equation}
\newtheorem{defn}[thm]{Definition}
\newtheorem{remk}[thm]{Remark}
\numberwithin{equation}{section}

\begin{document}

\title{Conformal Automorphism Groups, Adapted Generating Sets and Bases}
\author{Jane Gilman}%

\address{Mathematics and Computer Science Department, Rutgers University, Newark, NJ 07102}%

\email{gilman@rutgers.edu, jgilman@math.princeton.edu}%

\thanks{Some of this work was carried out when the author was partially supported at ICERM, Princeton University, at the NSF as an IPA  or was  visitor at the CUNY Graduate Center.}


\keywords{conformal automorphism, Schreier-Reidemeister rewriting system, adapted generating sets, and  Riemann surfaces}%

\date{revision of \today}


\begin{abstract} Let $S$ be a compact Riemann surfaces of genus $g \ge 2$ and $G$ a conformal automoprhism group of order $n$ acting on $S$. In this paper we give the definition of an adapted generating set and an  adapted  basis for the first homology group of such a compact Riemann surface. This generating set and basis reflect the action of $G$ in as simple manner as possible. This can be seen in the matrix of the action of $G$ which we obtain. We prove the existence of such a generating set and basis for any conformal group acting on such a surface and find the matrix. This extends our earlier results on adapted bases and matrices for automorphism groups of prime orders and other specific groups.
\end{abstract}


\maketitle
\tableofcontents



\newpage

\section{Introduction}
Our goal here is to derive results about compact Riemann surfaces with conformal automorphism groups. In particular we will define the notion of an adapted generating set for first homology group and an adapted basis, ones that are  adapted to the conformal automorphism group. We  prove the existence of such a generating set and basis for any given group of finite order that acts on a compact surface of a given genus at least two and find the matrix of the action. The purpose of the adapted generating set and adapted basis is to reflect the action of the group in as simple a manner as possible. The original goal was to use adapted bases to identify the Jacobians corresponding to surfaces with certain conformal automorphisms and to produce the structure of a quasi-projective variety on the moduli space of punctured surfaces.  Here we extend results and definitions in \cite{thesis,Moduli, JG1, matrix,Jalg, India}.

We note that the results  and their  proofs can be obtained using various equivalent techniques, namely the technique of Schreier-Reidemeister, curve lifting, or fundamental polygons and surface symbols. Often
results from these techniques are used in combination. Here we concentrate on  the Schreier-Reidemeister theory as it seems to give the cleanest proofs once the notation is set.  The notation needed can be cumbersome and some results obtained from other techniques are used at times.

\section{Method Overview} 
We determine  a set of curves on a genus $g$ surface,  $S$,  that generate the fundamental group of $S$ together with the action of the elements of a finite conformal automorphism group, $G$,  on these curves and the action on the abelianization of the fundamental group,  that is, the action on an integral basis for the first homology. First applying the Schreier-Reidemeister Method (section 2.3 of \cite{MKS}),
to our situation we  obtain generators and relations for $\Gamma$, the uniformizing group of $S$. These are more generators and relations than needed. Thus we eliminate some generators and relations to obtain a single defining relation for $\Gamma$ and $2g$ generators. We compute the conjugations action of $G$ on the generators and the matrix of the action of $G$ on the abelianized fundamental group, the integral homology group.  As noted above the proofs of the theorems using the Schreier-Reidemeister  may be preferable to other techniques because such proofs use a less ad hoc method.

 Our main result is to find the adapted generating set and homology basis and homotopy bases for a conformal group given a surface and its surface kernel map or equivalently its generating vector and to use the results to describe the action of the group on the generating set and the basis.


\subsection{Further Details on Method}

We work with the uniformizing groups $\Gamma$ and $\Gamma_0$ of the surface $S$ and the quotient surface $S_0$.  
We let $\phi$ be the surface kernel map, so that $\phi(\Gamma_0)=G$ and $\Gamma$ is what is known as the surface kernel. We then calculate the action of the automorphism group by conjugation on the surface kernel.
The elimination of generators  and relations to one relation is done in a manner that first allows us to compute the action of $G$ on the fundamental group. Usually this will be a cumbersome action, although one can and does write it down. One consider the image of the action on the first homology, the abelianization of the fundamental group. This is less complicated to write down and often equally useful. We  find a generating set for homology on which the action is fairly simple and a homology basis on which the action is relatively simple.

We use the generating vector or the equivalent surface kernel map,  appropriate automorphisms of $\Gamma_0$ and $G$, a well-chosen set of minimal Schreier representatives and a scheme for the elimination of generators and relations to obtain Theorem \ref{thm:SR}. This is  similar to the proof in the prime order case. The next main result uses a reduction and elimination process applied to the result of Theorem \ref{thm:SR}  to obtain  the single defining relation and set of $2g$ generators for $\Gamma$ (Theorem  \ref{thm:Algorithm}) and finally by computing the action of $G$ on $S$ we obtain the theorems on adapted generating sets and basis (Theorems \ref{thm:adaptedHOM}  and \ref{thm:MAB}).
\subsection{Organization} \label{sec:organization}

This paper consists of three main parts. Part 1 involves elimination to one relation and $2g$ generators. Part 2 involves the action
of $G$ on $\Gamma$, the kernel. Part 3 involves adapted homology generating sets and bases, their  definitions and existence.
The main results including the Definitions  \ref{def:AHB} and \ref{def:adaptedHomologyBasis}and Theorems \ref{thm:adaptedHOM} and  \ref{thm:MAB} appear in sections \ref{sec:adapted} and \ref{sec:reductionTObasis}.

Part \ref{part:ONE} begins by fixing notation (section \ref{sec:notation}), next  it reviews coverings and the  Riemann-Hurwitz count,  and then it reviews the Schreier-Reidemeister theory (section \ref{sec:SRSum}). It concludes  with the application of  the Schreier-Reidemeister theory to our situation (section \ref{sec:SR}) to obtain a set of generators and relation for the normal subgroup, $\Gamma$ of $\Gamma_0$,  in  Theorem \ref{thm:SR} and Corollary \ref{cor:numbers}. Part \ref{part:TWO} is a section on elimination. In section \ref{sec:redElimProc} and the following sections in this part  we  reduce the presentation for the normal subgroup to a set of $2g$ generators and a single defining relation (Theorem \ref{thm:Algorithm}). In Part \ref{part:THREE} we compute the action of the group $G$  on the normal subgroup $\Gamma_0$  and we define  of an adapted generating set for homology and adapted homology basis and prove their existence. Theorem \ref{thm:adaptedHOM},  Theorem \ref{thm:MAB} obtaining the matrix of the action as a corollary.
\vskip .1in


\part{Preliminaries:  Notation, Background and First Result: Application of the Shreier-Reidemeister Theory} \label{part:ONE}

\section{Notation} \label{sec:notation}
Let $S$ be a compact surface of genus $g$, $G$ a conformal automorphism group of order $n$ and $S_0 = S/G$ the quotient surface. Assume that $U$ is the unit disc and that $\Gamma$ and $\Gamma_0$ are Fuchsian groups with $S=U/\Gamma$ and  $S_0=U/\Gamma_0$ where $\Gamma/\Gamma_0$ is isomorphic to $G$. $\Gamma$ and $\Gamma_0$ are, of course, isomorphic to the fundamental groups of $S$ and $S_0$ with appropriate choice of base points. Let $\phi: \Gamma_0 \rightarrow G$ be the surface kernel homomorphism so that $\Gamma$ is the kernel of $\phi$.
. Knowing the isomorphism between the fundamental group of $S$ and $\Gamma$, we speak about the action of $G$ on $S$ meaning the  equivalent action of $\Gamma_0$ on $\Gamma$ which is given by conjugation. We may assume that $\Gamma_0$ has presentation given
by

\begin{equation} \label{eq:presentation}
\Gamma_0=
\langle a_1,...,a_{g_0}, b_1, ..., b_{g_0}, x_1,...,x_r  \; |\; (\Pi_{i=1}^{g_0} [a_i,b_i])x_1 \cdots x_r =1; \;\; x_i^{m_i} =1 \rangle
\end{equation}

Here the $m_i$ are positive integers each  at least $2$ and $[r,s]$ denotes the multiplicative commutator of $r$ and $s$. Any elliptic element of the group is conjugate to one of the $x_i$. No two different $x_i$'s are conjugate.

Knowing the surface kernel map is equivalent to knowing the generating vector, the vector of length  $2g_0 +r$
given by $$(\phi(\alpha_1), ... ,  \phi(\alpha_{g_0}),\phi(\beta_1),..., \phi(\beta_{g_0}), \phi(x_1), ..., \phi(x_r)).$$

Let $\pi: S \rightarrow S/G=S_0$. If $p$ is fixed by an element of $G$, then $G_p$, the stabilizer of $p$ is cyclic, say of order $m_p$.
We let $p_0= \pi(p)$. Then $\pi^{-1}(p_0)$ consists of $n/m_p$ distinct points each with a cyclic of stabilizer of order $m_p$.  The point $p_0 \in S_0$ is called a {\sl branch point} and $p$ is called a {\sl ramification point}. If the covering has $r$ branch points of order $m_1,...,m_r$, then over $p_0$ instead of $n$ points, there are ${\frac{n}{m_p}}(m_p-1)$ points. That is, at each point over $p_0$ we are missing $m_p-1$ points so there is a total of ${\frac{n}{m_p}}(m_p-1)$ missing points. Over $p_0$ there are ${\frac{n}{m_p}}$ ramification points.

The presentation \ref{eq:presentation} is equivalent to there being are $r$ branch points of order $m_i, i=1,...,r$, and thus by the Riemann-Hurwitz relation we have

\begin{equation} \label{eq:RH}
2g-2 = n(2g_0 -2) + n \cdot \Sigma_{i=1}^r (1- {\frac{1}{m_i}})
\end{equation}

We note for future use that if one has eliminated generators and relations and one has exactly one relation in which every generator and its inverse occurs exactly once, then it is equivalent by a the standard algorithm in \cite{Springer} to a standard surface presentation. The algorithm in \cite{Springer} is given geometrically in terms of cutting an pasting sides of a fundamental polygon but translates to a purely algebraic algorithm (see \cite{Jalg}). A standard surface presentation is one given by $2g$ generators and one relation which is the product of commutators.

We note that in using \ref{eq:presentation} we may assume by the Poincar{\'e}{\'e} Polygon theorem \cite{Beard, Maskit} certain angle conditions at the vertices of the elliptic elements are satisfied. Further, the relation of the presentations for $\Gamma_0$ and $\Gamma$ to the branch structure allows us to assume that the order of an elliptic generator  and that of its image in $G$ are the same.

We use $\approx$ to denote that words in  a group are freely equal and
$\approx^h$ to denote that two curves are homologous, that is equal  as elements of an integral homology basis and we write the integral homology additively. For simplicity we use the same notation for a curve, its equivalence class as an element of the fundamental group and its image in homology.

\section{Background: Summary of the Schreier-Reidemeister theory} \label{sec:SRSum}
In this section we first recall standard terminology, facts and theorems for the Schreier-Reidemeister Theory (see section 2.3 of \cite{MKS}.)
Next we choose a set of Schreier right coset representatives for $\Gamma_0$ modulo $\Gamma$ and denote an arbitrary representative by $K$. We let $\tau$ be the rewriting system so that the generators of $\Gamma$ are given by $S_{K,a} = Ka{\overline{KA}}^{-1}$ where $\overline{X}$ denotes the coset representative of $X$.  Our first result is comes from the application of the Reidemeister-Schreier theorem to our situation (Theorem 2.9 page 94 \cite{MKS}). It would be nice to simplify the notation for this,  but I have yet to find a better notation.
\subsubsection{Schreier Representatives}

Let $G$ be a finite group of order $n$,  $\Gamma_0$ a group with known presentation, $\phi$ a group  homomorphism of $\Gamma_0$ onto  $G$ and $\Gamma$ the kernel of $\phi$.  We let $K_1, ..., K_n$ a set of right coset representatives for $\Gamma_0$ with $K$ denoting an arbitrary one of these chosen right-coset representative.

\begin{defn}

A  {\sl Schreier right coset function} is a right coset function  where the initial segment of any coset is again a right coset representative. It is a minimal if the length of any coset does not exceed the length of any coset it represents.
 We call the set of cosets a {\sl Schreier} system and denote the coset of word $W$ in $\Gamma_0$ by ${\overline{W}}$.
\end{defn}

Note that we can always choose a set of {\sl minimal} Schreier representatives, a set where each representative is of minimal length in its coset.

A rewriting process is  process that takes a word that is in $\Gamma$ but that is given in the generators of $\Gamma_0$ and writes it as a word in generators for $\Gamma$. A {\sl Schreier-Reidemeister rewriting process} is one that uses a Schreier system. Here we denote our rewriting process by $\tau$.

\subsection{The rewriting system $\tau$}

Assume $\Gamma_0$ has presentation given by generators $w_1,...,w_q$ with relations $R_u(w_1,...,w_v)$ for some integers $q$, $u$  and $v$ where the $R_u$ are words in the generators.

We remind the reader that a Schreier-Reidemeister rewriting process $\tau$ writes a word in the generators of $\Gamma_0$ that lies in $\Gamma$ in terms of the generators $S_{K,v}$ for $\Gamma$ where $K$ runs over a complete set of Schreier  minimal coset representatives. 

We let ${\overline{X}}$ denote the coset representative of $X$ in $\Gamma_0$. That is, the element $K$ where $\phi(K) = \phi(X)$. The element $S_{K,v} \in \Gamma$ is the element $Kv{\overline{Kv}}^{-1}$.

The rewriting process is defined as follows:

Let $X= a_1^{\epsilon _i} \dots a_u^{\epsilon_u}$, where each $\epsilon_i$ is either $+1$ or $-1$.
Then $$\tau(X) = \Pi_{i=1}^r S_{V_i,a_i}^{\epsilon_i},$$
 where $V_i$ depends upon $\epsilon_i$. Namely, if $\epsilon_i = +1$, then $V_i = {\overline{a_1^{\epsilon_1} \dots a_{i-1}^{\epsilon_{i-1}}} }$ and if
$\epsilon_i = -1$, then $V_i = {\overline{a_1^{\epsilon_1} \dots a_{i-1}^{\epsilon_{i-1}} a_i^{-1}}}$.


\section{Application of the Schreier-Reidemeister Theorem} \label{sec:SR}

Apply  the Schreier-Reidemeister Theorem to our situation to obtain:

\begin{thm} 
\label{thm:SR} {\rm(Presentation for $\Gamma$ with Schreier generators)}

Let $\Gamma_0$ have generators
 $$ a_1,...,a_{g_0}, b_1,...,b_{g_0}, x_1, ...x_r $$
and relations
 $$ R=  \Pi_{i=1}^{g_0} [a_i,b_i]  \cdot x_1 \cdots x_r =1,    \;\;\; x_i^{m_i}=1$$

and let $\Gamma$ be the subgroup of $\Gamma_0$ with $G$ as above isomorphic to  $\Gamma_0 / \Gamma$.

The $\Gamma$ can be presented with generators 

$$  S_{K,a_i}, \;\; S_{K,b_i}, i =1,...,g_0,   \mbox{ and } \;\; S_{K,x_j},  j=1,...,r$$
and relations
$$\tau(KRK^{-1}) = 1,\;\; \tau(Kx_j^{n_j}K^{-1})=1,\;\; j = 1,...,r, \;\; S_{M,v}=1$$
where $M$ is a Schreier representative and $v$ is a generator such that $Mv={\overline{Mv}}$
\end{thm}
\begin{cor} \label{cor:numbers}
If $\tau$ is a Reidemeister-Schreier rewriting process, then $\Gamma$ can be presented as a group with $2ng_0 + nr$  generators and $n + (n-1) + \Sigma_{i=1}^r {\frac{n}{m_i}}$ relations.
\end{cor}

We refer to the $\tau(KRK^{-1})$ relations as the $R$-relations, the $\tau(Kx_j^{n_j}K^{-1})$ as the elliptic or $E$-relations and the $S_{M,v}$ as the $M$-relations. 

We have $2g_0n + nr$ generators, which we refer to as the Schreier generators and further distinguish between the hyperbolic Schreier generators and the elliptic Schreier generators, meaning the images of hyperbolic or elliptic elements of $\Gamma_0$. The Schreier generators are, of course, all hyperbolic transformations but we sometimes refer to these as $H$-Schreier,  $E$-Schreier and $N$-Schreier generators
 or $H$-generators, $E$-generators and $M$-generators for short.
 There are $n$  $R$-relations,  $nr$ elliptic relations,  and $n-1$ $M$-relations.

We note that the action of $G$ or $\Gamma_0$ on $\Gamma$ is given by conjugation and if we set $\phi(K) = g_K$, then
for any word $W$ in $\Gamma$.\

\begin{eq} \label{eq:theGaction}
$g_K(W) = \tau(KWK^{-1})$
\end{eq}

This will be used in section \ref{sec:Gaction}.


Our next goal is to show that one can perform Teitze transformations to obtain a presentation for $\Gamma$ with $2g$ generators and one relation that involves each generator and its inverse exactly once.
\vskip .05in
\vskip .1in
\part{Reduction and Elimination } \label{part:TWO}

\section{Reduction and Elimination Procedures \label{sec:redElimProc}}



\vskip .05in

We use elimination procedures to find a presentation of $\Gamma$ with $2g$ generators and one relation where each generator and its inverse occur exactly once.  The types of elimination procedures are defined below and are termed respectively $M$-elimination, elliptic elimination and gluing  elimination, for short $M$,$E$ and $Gl$ elimination. $E$ and $Gl$ eliminations are Teitze transformations of type $T4$ in the terminology of \cite{MKS} (See page 50.)

The order in which we use the different type of elimination steps will be described shortly as well as the count for the number of generators and relations remaining after given combinations of steps. First we choose a special set of coset representatives.


\section{Choosing the  Schreier coset representatives}

We could use what is known as {\sl short-lex order} on the elements of $\Gamma_0$ and choose for coset representatives the $n$ shortest words with this order. We {\sl modify}  this by omitting the $x_i^{-1}$ from our list when  $x_i^{m_i} =1$, as $x_i^{-1}=x_i^{m_i-1}$ will appear.

We carry out the following steps:

\begin{enumerate}

\item After applying an automorphism of $\Gamma_0$, we may assume that the orders of the elliptic elements satisfy $m_1 \le m_2 \le \cdots \le m_r$.
\item We order the $2g_0+r$ generators and the inverses of the hyperelliptic generators as follows:
$$ x_1 <  x_2<   ...<  x_r < a_1 <  a_1^{-1}< a_2< a_2^{-1}< \cdots< a_{g_0} < a_{g_0}^{-1} < b_1< b_1^{-1}<  \cdots < b_{g_0}<b_{g_0}^{
-1}.$$

\item We impose lexicographical order on the generators and choose for the coset representatives $K$, the $n$ first elements.
\end{enumerate}

\subsection{Elliptic Elimination}
If $x$ denotes an elliptic generator of order $m$, then the equation that governs elliptic is

\begin{eq} \label{eq:elliptic} $\tau(x^m) = S_{1,x}S_{{\overline{x}},x} \cdots
S_{{\overline{x^{m-2}}},x}S_{{\overline{x^{m-1}}},x}$.
\end{eq}

We can solve for $S_{{\overline{x^{m-1}}},x}$ and replace it wherever it occurs in an $R$-relation by $(S_{1,x}S_{{\overline{x}},x} \cdots
S_{{\overline{x^{m-2}}},x})^{-1} =1$.
This eliminates one generator and one relation and replaces an $R$-relation by a modified relation. For ease of exposition we do not use a separate notation for the modified $R$-relation.

Since for the images of every hyperbolic Schreier generator,  every generator and its inverse occurs exactly once in the set of original $R$-relations, the same is true for the remaining generators and the modified $R$-relations. However, we also have that every remaining $E-$Schreier generator and it inverse occurs.

Elliptic elimination will reduce the number of generators by $\Sigma_{i=1}^r {\frac{n}{m_i}}$ and the number of relations to $n$.

We note that the images of a given  $x_t$ that appear in the $\tau(KRK^{-1})$ are in distinct $r$-relation and that exactly one appears in each $R$-image.

\subsection{Elimination of $R$-relations: Gluing} \label{section:Relimination}

\begin{defn} \label{defn} {\rm (Gluing)}  Let $R_1$ and $R_2$ be relations in $\Gamma$. If we have $R_1 = W_1XW_2$ and $R_2 = V_1X^{-1}V_2$,  then we can solve for $X$ as $(V_1^{-1}R_2V_2^{-1})^{-1}$, eliminate $X$ from the list of generators and replace
$R_1$ by the relation $W_1V_2R_2^{-1}V_1W_2$ where $V_1, V_2, W_1, W_2$ are words in the Schreier generators for the kernel $\Gamma$  and no pair of these four words contain  a common Schreier generator. This process is termed 
{\sl sewing $R_1$ and $R_2$ along $X$}
or {\sl gluing $R_1$ and $R_2$ along $X$.}\end{defn}
 A gluing operation eliminates  $X$ and $X^{-1}$. Thus it eliminates one $R$-relation and one Schreier generator. We call $X$ a {\sl gluing generator}   and say that we {\sl have glued along $X$}.
\subsection{$M$-elimination}

Since the $M$-generators are $\approx 1$, they simply drop out of the elimination procedure. However, we will at some point to identify more precisely where and what they are. A procedure for this is given in section \ref{sec:Mremove}.

We also  note that the $M$-generators can be found from the spanning tree of the Cayley diagram.
\subsection{The Order for the Eliminations}

We use the following order:
\begin{description}

  \item[Step 1] Carry out an {\sl elliptic elimination} for those of the ${\frac{n}{m_i}}$ elliptic relations for each of the elliptic relations, $i=1,...,r$.
  \item[Step 2] Carry out gluing $n-1$ times as described below.
\item[Step 3] Replace every $M$-generator by $1$ in our calculations.

\end{description}

{\sl Gluing elimination order:}
Begin with $\tau(R) = R_1$ and consider the first element $Q$ that appears in $R_1$. We let $R_2$ be the relation in which $Q^{-1}$ occurs and glue $R_1$ and $R_2$ along $Q$ to obtain $\tilde{R_1}$. We consider the first element $\tilde{Q}$ of $\tilde{R_1}$ and locate $R_3$ the relation where $\tilde{Q}^{-1}$ occurs and sew $\tilde{R_1}$ and $R_3$ along $\tilde{Q}$ and continue until we have one relation left.
We assume that the gluing is not along an $M$-generator and prove below (Lemma \ref{lem:enough}) that there are enough generators to justify this assumption.

\begin{remk} \label{rem:refinement}{\sl Refinement for later refinement we can :}
 begin with $R_1$ and find the first $S_{1,v} \ne 1$ that occurs and consider that $S_{K,v}^{-1}$ that also occurs. Find the coset representative ${\tilde{K}}$ where $S_{K,v}$ occurs in $\tau({\tilde{K}}{\tilde{K}}^{-1})$ and glue this to $R$. Continue until this closes up. The number of copies will be some integer $w$ that divides $n$, the order of $G$. When we have repeated this until there is one relation left, the product of the number of relations glued at each step will be $n$.
\end{remk}

\subsection{Gluing Cycles}
 We consider the generators for $\Gamma$ after elliptic elimination and $M$-elimination and the $n$ modified $R$-relations.

We can glue in a manner analogous to the determination of cycle conditions in the Poincare  Polygon theorem \cite{Beard, Maskit}. Namely,begin with $\gamma = S_{1,v_1} \in \tau(R)$ which has not been deleted, $v_1$ some generator of $\Gamma_0$.  Let $K$
be such that $S_{1,v_1}^{-1}$ appears in $\tau(KRK^{-1})$ and let the other image of $S_{1,v_1}$ that appears in $\tau(KRK^{-1})$ be $S_{L,a}$ where $L$ is one of the coset representatives. Write $\gamma=S_{1,v_1}$ and let $h_1 \in G$ be the element with $h_1(\gamma)=S_{L,v_1}$.

Find $h_2(\gamma)$ in the same manner. Obtain is a sequence $h_1, h_2 \cdots,  h_t$ of elements of $G$ such that $h_1 \cdot h_t(\gamma) = \gamma$. We have also obtained a corresponding cycle of $R$-relations. We glue along this sequence to obtain $\tilde{R}$. We have $t= m_i$ for some integer $m_i$.  The there are coset ${\frac{n}{m_i}}$ coset representatives $J$ where the $\tau(J\tilde{R}J^{-1}$ which are all distinct. $g_K \ne h_i$  for any $i$ we obtain and also along the images of these. That is we take $k_1$ a coset representative such that $\tau(k_1Rk_1^{-1})$ is not in the cycle of $R$-relations. We note if $\phi(a) = m$, then we have ${\frac{n}{m}}$ glued relations left. We pick the next $v$ such that $S_{{\tilde{K}},v}$ appears in $\tau(R)$,  $\langle \phi(v) \rangle \cap \langle \phi(\gamma) \rangle = \emptyset$ and  apply the procedure to the ${\frac{n}{m}}$ clusters.
If we repeat this procedure $z$ times each gluing $m_z$ copies of $R$-clusters, we must have $\Pi_{i=1}^z m_i= n$ and we are done.

Note each $m_i$ divides $n$. Need to show there are enough such $v$ left. This requires a coset argument.

It follows from the fact at each step we remove the inverses that occur that cycle gluing eliminates ${\frac{n}{m_i}}$ generators.

We have \begin{lem}
Gluing along cycles reduces the $n$ relations to one relation where every generators and its inverse occurs. If $v$ is an E-Schreier generator of order $m$, $S_{1,v_1}$ all of its images $S_{{\overline{Kv_1^q}},v_1},  q=1,..., m-1$ occur. If $v$ is an $H$
-Schreier generator, with $(\phi(v))^m =1$,
all of its images $S_{{\overline{Kv_1^q}},v_1} q=1,..., m-1$ which are not $M$ generators or among the eliminated $E$-generators occur.\end{lem}

We need a lemma to show that we have enough generators remaining after $M$ and $E$ elimination to be able to carry out the needed $n-1$ eliminations. This follows from Lemma \ref{lem:enough} below.

\subsection{$M$-elimination} \label{sec:M-elim}

The $M$-Schreier generators simply disappear from the relations, but we will need to keep track of where they occur, how many of them there are, and whether they are elliptic Schreier generators or hyperbolic ones. In section \ref{sec:Mremove} below we present a procedure for carry out the elimination of $M$-generators which allows us to determine the action of $G$ on remaining generators.

\section{Removing the $M$-generators} \label{sec:Mremove}

We note that the $M$-generators can all be found by considering the spanning tree of the Cayley graph. Here we present a procedure for what we term the  {\sl $M$-removal} of the $M$-generators.
First, we consider the $E$-Schreier generators that are $M$-generators.

We use the notation $S_{*,v}$ when the first subscript is a coset representative which we do not need to specify.
In $\tau(KRK^{-1})=1$ let $K = w_1 \cdots w_t$.

Then the generators $S_{*,w_j}^{\pm 1}$ that appear in the image of $\tau$ are $M$-generators, so they drop out of the expression and the image under $\tau$ ends with an $S_{*,x_r}$.

If any for any of the $*$  these words are $M$-generators,  they are removed and we can solve for $S_{*,x_{r-1}}$.

For those which are not $M$-words we solve in terms of the remaining generators. Those which are $M$-generators will drop out and we will solve for $S_{*,x_{r-2}}$.

We continue the procedure until we obtain a Schreier generators that is an $M$-generators. Reaching the image of $[a_{g_0},b_{g_0}]$ in the $\tau(KRK^{-1})$ does not present a problem for we consider the subword,

$S_{{\tilde{K}},a_{g_0}}S_{{\overline{{\tilde{K}}a_{g_0}}},b_{g_0}}
S_{{\overline{{\tilde{K}}a_{g_0}b_{g_0}a_{g_0}^{-1}}},
a_{g_0}}^{-1}S_{{\overline{{\tilde{K}}[a_{g_0},b_{g_0}]}},b_{g_0}}^{-1}$
where ${\tilde{K}}$ is the appropriate coset representative.

We can solve for $S_{{\overline{{\tilde{K}}[a_{g_0},b_{g_0}]}},b_{g_0}}^{-1}$ 
and its images under the group and omit those which are $M$-words and replace those which are not $M$-words by their solutions. We can compute the action of $G$ as needed.
We continue until all $M$ words have been eliminated.

At the end of this procedure, we now have exactly $2g$ generators and a $2g \times 2g$ matrix. The action of$G$ on the rows from which the $M$ generators  have been removed, which we term {\sl $M$-rows} of the matrix is more complicated than in the rows of the adapted generating set,  but only affect at most $n$ lines. This is because some $M$-rows may come from removing more than one $M$-generator.
That is,  we obtain a $2g \times 2g$ matrix with at most $n$ lines altered.

Sometimes computing the $2g \times 2g$ matrix from the $(2g + n) \times (2g+n)$ matrix is easier, for example when there is a maximal power Schreier generator. A maximal power generator is one for which
${\overline{v^q}}=v^q \forall q with 1 \le q \le m_{v-1}$,
 where we define $m_v$ to be the order of $\phi(v)$ whether $v$ is elliptic or hyperbolic. (See section \ref{sec:easier})


\begin{remk} Further, if there are no elliptic generators with the same cyclic group generated by the image as a hyperbolic generator $v$ we will see that this will also give us two words in the kernel that are fixed under the element of $G$, namely $\phi(v)$.
\end{remk}

\subsection{The Count} \label{sec:thecount}

\begin{lem} The elimination procedure leaves $2g$ generators.
\end{lem}
\begin{proof}
If there are no elliptic generators with ${\overline{x^{m-1}}} = x^{-1}$, then elliptic elimination removes $\Sigma_{i=1}^r {\frac{n}{m_i}}$ generators. After elliptic elimination we have
$2ng_0 + nr - n\Sigma_{i=1}^r {\frac{n}{m_i}}$. We remove a further $n-1$ $M$-generators and glue along $n-1$ generators. We are left with $2ng_0 + nr - n\Sigma_{i=1}^r {\frac{n}{m_i}}$. This is precisely $2g$ generators.

We need to check that after elliptic elimination and $M$-elimination, there are still enough generators to remove $n-1$ more.

\end{proof}

\begin{lem} \label{lem:enough} {\rm (Enough Generators)}
There enough generators for $\Gamma$ to allow elimination of another $n-1$ generators after elliptic $E$-elimination and $M$-elimination have been carried out.
\end{lem}

\begin{proof}

We have Riemann Hurwitz formula again:
$$2g-2 = n(2g_0 -2) + n \cdot \Sigma_{i=1}^r (1- {\frac{1}{m_i}}).$$
We have eliminated the correct number for the elliptic generators and $M$ elimination: that is  $n \cdot \Sigma_{i=1}^r (1- {\frac{1}{m_i}}) + n-1$ or $n \cdot \Sigma_{i=2}^r (1- {\frac{1}{m_i}}) + n-1$.
depending upon whether  an elliptic or hyperbolic generator has been removed.

We want to show the remaining $2g-n \cdot \Sigma_{i=1}^r (1- {\frac{1}{m_i}}) = 2 + n(2g_0 -2) \ge 2n-1$.
 Since we have also removed $n-1$ of the $M$-generators, we need $2 + 2ng_0 -2n -n \ge n-1$
or  $2 + n(2g_0) -3) \ge n-1$.
 If $g_0 \ge 2$, $2g_0-3 \ge  1$, we are OK.  
 If $g_0=0$, for there to exist a Fuchsian group of genus $0$ with $r$ elliptic elements,it must be that  $-2 +\Sigma_{i=1}^r(1-{\frac{1}{m_i}}) >0$.
Since $m_i \ge 2$, it must be that $r \ge 3$ and the needed equation holds.
If $g_0 =1$,  we must have $r \ge 1$ because we must have $\Sigma_{i=1}^r(1-{\frac{1}{m_i}}) >0$ for such a Fuchsian group to exist, again the needed equation holds.

\end{proof}

We will later need to refine this count to see where the removed  $M$ generators occur so as to have more information about the action of $G$ on the kernel.

\begin{lem} \label{lem:beta-alphaNOTequal}
After applying an automorphism if necessary, we may assume $\forall i=1,..., g_0, {\overline{\alpha_i}} \ne {\overline{\beta_i}}$. Then each $S_{K,a_i}$ and its inverse occur in different $\tau(KRK^{-1})$ where $K$ runs over all coset representatives, $K_i, i=1, ..., n$.
In a similar manner we may assume applying an automorphism is necessary that in the gluing process we do not need to use a commutator in which $\phi(b)=1$
\end{lem}

\begin{proof}

 We note that for gluing that at each step the gluing  generators $U$ and $U^{-1}$ appear in different relations. Namely, if there is an element $U$ such that $UVU^{-1}\tilde{V}) = \tau(qbu^{-1}v^{-1})$ appears as a subword of $\tau(KRK^{-1})$ for some $q,b u,v
\in \Gamma_0$,  then $\phi(q)$ is conjugate to $\phi(v)$ in $G$ and we can apply an automorphism of $\Gamma_0$ to $\Gamma_0$ to assure that this is not the case. (See \cite{Harvey} or  \cite{BW} for lists of automorphism.)

Equivalently we may assume that $U$ and $U^{-1}$ do not appear in the same $\tau(KRK^{-1})$ for any $K$. This is clearly true for any elliptic Schreier generator. Consider all possible ways in which

$S_{K,a}=S_{{\overline{Kaba^{-1}}},a}$
or


$S_{K,b} = S_{{\overline{K[a,b]}},b}$

could occur for some $K$ and note that under any of the circumstances, no matter whether or not $G$ is abelian, we can apply an automorphism to $\Gamma_0$ to rule these out.
See the lists of automorphisms in \cite{Harvey1, Harvey2, BW}.

Again use the automorphisms list, if need be, to replace the generators by generators under the automorphism $Z_i$ as long as ${\overline{\beta_{i+1}}} \ne {\overline{\beta_i}}$ when $r=0$ otherwise use the automorphism $U_{ij}$.

\end{proof}

\begin{remk}
Using  $\gamma= S_{1,v} $ and its images under $G$ is equivalent to using than $S_{K,v}$ because as $K$ varies over all coset representatives $g_K(\gamma)$ varies over all $S_{K,v}$.
\end{remk}
\begin{remk}
While one can replace the final single relation by one involving  $2g$ generators for the fundamental group and the product of their commutators, we do not do so here. The bases and generating sets we find here  are not ones with that intersection matrices although it can be computed using the method initiated by D. Patterson \cite{GP,Jalg} and  later applied in extended circumstances by \cite{Rubi}.
\end{remk}

\section{Elimination Algorithm} \label{sec:elimalg}

We note that we have shown:

\begin{thm} \label{thm:Algorithm}
The application of the Schreier-Reidemeister theorem together with the elimination procedures and their orders of execution given above gives an algorithm whose output is $2g$ generators for $\Gamma$ and a single defining relation in which every generator and its inverse occurs.
\end{thm}

\part{Action of $G$ on $\Gamma$}   \label{part:THREE}

\section{$G$ action} \label{sec:Gaction}


In this part we study the action of $G$ on the kernel. Namely if $K$ is an arbitrary coset representative we let $g_K \in G$ denote $\phi(K)$ and note that $g_K(W) = \tau(KWK^{-1})$ where $W$ is any word in $\Gamma_0$ that lies in $\Gamma$. If $v$ is any generator for $\Gamma_0$,  we also write
 $g_{Kv}$ for $g_{\overline{Kv}}$.   We simplify the notation by writing since  it is understood that we mean the coset representative of $Kv$ by its position as the first subscript.

We want to compute $g_W(S_{K,v})$ for any coset representatives $W$ and $K$ and any generator $v$ of $\Gamma_0$. To do so, we need the concept of {\sl tails}.

\subsection{Action of $G$ Tails and Prefixes and Orders of Hyperbolic Generators}
Consider $S_{1,v}$ where $v$ is a generator for $\Gamma_0$ and that so that ${\overline{v}}= w$ is of length one. Assume that $K=w_1 \cdots w_t$ is a minimal
 Schreier right coset representative for some integer $t$.
We calculate \begin{eq} \label{eq:tails1} $g_K(S_{1,v})= \tau(KS_{1,v}K^{-1}) =
\tau(w_1 \cdots w_tv{\overline{v}}^{-1}(w_1 \cdots w_t)^{-1}) = \tau(w_1 \cdots w_tvw^{-1} (w_1 \cdots w_t)^{-1}) =S_{K,v}S_{K,{\overline{v}}}^{-1}$.
\end{eq}
We call $S_{K,{\overline{v}}}$ {\sl the tail of $S_{K,v}$
and denote it by $T_{K,v}$. We have

\begin{eq}$g_K(S_{1,v}) = S_{K,v}T_{K,v}^{-1}$ \label{eq:gTail}
\end{eq}

\begin{defn} We define the order of a hyperbolic generators $v$ of $\Gamma_0$ to be $m$ if $\phi(v)$ has order $m$.
\end{defn}
Thus all elements of $\Gamma_0$ have orders and we denote the order of an arbitrary element, $v$,  by $m_v$.

We next consider  the equation
\begin{eq} \label{eq:act1} $\tau(v^m) = S_{1,v}S_{{\overline{v}},v} \cdots S_{{\overline{v^{m-1}}},v}=1$.
\end{eq}
We note that in this expression the tail of $S_{{\overline{v^t}},v}$ will be the inverse of
the {\sl preface} the segment that occurs after the tail and before $S_{{\overline{v^{t+1}}},v}$.

More generally, we can define the Tail of $g_L(S_{K,v})$ by setting $g_L(S_{K,v}) = S_{{\overline{LK}},v}T_{{\overline{LK}},v}$ and note that tails and prefixes in computing
the image of any relation under $\tau$ will cancel out.

Regardless of whether $S_{1,v}$ is an $E$ or an $H$ generator considering the action on  the images in homology, we have
\begin{lem} \label{lem:homology-one} Let $\gamma= S_{1,v}$ where $v$ is any generator of $\Gamma_0$ and let
$g_v= \phi(v)$. We have $\tau(v^m) \approx^h  \gamma + g_x(\gamma) + \cdots + g_x^{m-1}(\gamma)$.
\end{lem}


In a similar manner, tails and prefixes cancel in $\tau(R)$. If we write $\gamma_i$ and $\delta_i$ and $y_j$ respectively for the Schreier words in $\tau(R)$ that are the images of the $a_i$, $b_i$ and $x_j$ $1=1,...,g_0$ and $j =1 ,..., r$, we can write have
\begin{lem} \label{lem-Rhomology}
$\tau(R) = (\Pi_{i=1}^{g_0} [\gamma_i,\delta_i])y_1 \cdots y_r$.

$\tau(R) \approx^h (\Sigma_{i=1}^{g_0} [\gamma_i,\delta_i]) + y_1 + \cdots  +y_r \approx^h y_1 + \cdots  +y_r .
$
\end{lem}

\subsection{Action on elliptic images in $\Gamma$}

We want to compute $g_K(S_{L,v})$ for $L$, $K$ coset representatives and $v$ any generator of $\Gamma_0$.
Suppose $x$ is an elliptic generator with $x^m=1$, and ${\overline{x}}= v$ with $\phi(v) = \phi(x)=h$.
Because this is a minimal Schreier system, the length of  $v$ is also $1$.
We write $h_x$ or $h_v$ for $h$ and using equation \ref{eq:act1} since $S_{1,x} = xv^{-1}$,
\begin{eq} \label{eq:act2}
$
h(S_{1,x}) = \tau(vxv^{-1} v^{-1})= S_{1,v}S_{v,x}S_{\overline{vxv^{-1}},v}^{-1}S_{1,v}^{-1} S_{v,x}S{v,v}.$

\end{eq}

The first and last entries of \ref{eq:act2} are $1$ because they are each an $M$-generators. 
\begin{lem}
\label{lem:Gxaction}
Assume that $x$ is an elliptic generator of $\Gamma_0$ with $x^m=1$ and that $v \in \Gamma_0$ is the generator with     ${\overline{x}} = v$, $\phi(v)= \phi(x) = h$,  ${\overline{x^q}} = v^q$ for $1 \le q \le m$ , then $h^q(S_{1,x}) = S_{v^q, x}$  for all $q$.
\end{lem}
Let $\gamma = S_{1,x}$.
With the same hypothese as above, we  can rewrite equation \ref{eq:act1} as
\begin{eq}
$h^{m-1}({\gamma}) \approx^h -(\gamma + h(\gamma) + \cdots h^{m_2}(\gamma))$
\end{eq}
 For each coset representative $g_h$ of $G$ modulo $\langle h \rangle$, we  have
\begin{cor} \label{cor:composition} 
$$(g_h \circ h^{m-1})({\gamma}) \approx^h -(g_h(\gamma) + (g_h \circ h)(\gamma) + \cdots (g_h \circ h^{m-2})(\gamma)).$$
\end{cor}

Next we consider the action on the hyperbolic Schreier generators of $\Gamma$. We let $v$ by a hyperbolic generator of $\Gamma_0$ or order $m$ with $\gamma=S_{1,v}$. We have

\begin{eq}
$h^{m-1}({\gamma}) \approx^h -(\gamma + h(\gamma) + \cdots h^{m-2}(\gamma))$
\end{eq}

For each coset representative $g_h$ of $G$ modulo $\langle h \rangle$, we  have
\begin{cor}
$$(g_h \circ h^{m-1})({\gamma}) \approx^h -(g_h(\gamma) + (g_h \circ h)(\gamma) + \cdots (g_h \circ h^{m-2})(\gamma)).$$
\end{cor}


\subsection{Action on hyperbolic images in $\Gamma$}

We can say that every $S_{K,v}$ and its inverse appears when $v$ is an $H$-generator or equivalently all of the images of $S_{1,v}=\gamma$, that is  $g_K(\gamma)$,  appear or $g(\gamma)$ appears $\forall g \in G$.


\subsection{Elements with no fixed points}
If $g \in G$ has no fixed points, then there is a $v \in \Gamma_0$ with $\phi(v) = g$ and $\phi(v)^m=1$ for some $m$. Further there is no elliptic generator of $\Gamma_0$ whose image lies in $\langle g \rangle$.
We let $p$ be a prime dividing the order of $g$ and let $g_c \in G$ satisfy
$g_c^p=1$. Then we know from \cite{India} that $g_c$ fixes two curves on the surface.

We now need to address the problem if we do not have a generator $v$ that works, but rather a word in the generators.

  \begin{cor}
If $g \in G$ does not fix any points on the surface $S$, then there are at least two curves
fixed by $g$ in the basis and all of their translates under elements of the group  are fixed by conjugates of $g$.
\end{cor}


\section{Definition:  Adapted Generating Set for Homology  \label{sec:adapted} }

We generalize our original definition of an adapted integral homology basis first to an adapted generating set for the first homology group with respect to the pair $(\Gamma,G)$:

\begin{defn} \label{def:AHB} {\rm (Adapted Generating Set for Homology )}

Let $\gamma$ be an arbitrary curve in $\mathcal{B}$, an integral homology  basis
 for $S$,  and let $G$ be  a group of conformal automorphisms of $S$ of order $n$. The a set of $2g + n$ generators ${\mathcal{S_B}}$ for the integral homology  $\mathcal{B}$,  is {\sl adapted to} $G$ if for each $\gamma \in {\mathcal{S_B}}$ one of the following occurs:

 \begin{enumerate}
 \item \label{item:prop1}  $\gamma$ and $g(\gamma)$ are in ${\mathcal{S_B}}$ for all $g \in G$ and $g(\gamma) \ne \gamma$.

 \item $\gamma$ and $h^j(\gamma)$ are  in ${\mathcal{S_B}}$  for all $j=0, 1,...,m_{i-2}$ where $ h \in G$ is of order $m_i$, and
           \label{item:prop2}
      $$h^{m_{i-1}}({\gamma})  \approx^h  - (\gamma + h(\gamma) + \cdots  + h^{m_{i-2}}(\gamma)).$$
 Further
       for each coset representative, $g_h$,  for $G$ modulo $\langle h \rangle$, we have that
         $g_h(\gamma)$ and $(g_h \circ h^j)(\gamma)$ are in the set  for all $j=0, 1,...,m_{i-2}$ and
         $$(g_h \circ h^{m_{i-1}})({\gamma})  \approx^h  - (g_h(\gamma) + (g_h \circ h)(\gamma) + \cdots (g_h \circ h^{m^{i-2}})(\gamma)$$

     \item $\gamma = h^r(\gamma_0)$ where $\gamma_0$ is one of the curves in item \ref{item:prop2} above.

         \item \label{item:prop3} $g(\gamma) = \gamma$ for all $g \in G_0$, $G_0$ a subgroup of $G$ of order $m$.

          All of the other $n/m$ images of $\gamma$ under $G$ are fixed appropriately by conjugate elements or  by elements representing the cosets of $G/{G_0}$ and are also in $\mathcal{S_B}$.
\item Exactly $n$ of the curves $h^k(\gamma)$ determined by the pair $(h, \gamma)$ and the integer $k$ are null homologous
 where $\gamma$ is one of the curves above and $h$ is an element of $G$.
\end{enumerate}

\end{defn}

We have generators for $\Gamma$ with and one relation. For simplicity we use the same  notation for a generator and its image of the corresponding curve in first homology.

We have shown


 \begin{thm} \label{thm:adaptedHOM} Given $G$, $S$, $S_0$, $n$, $g$, and $r$,  there is an adapted generating set for the first homology group containing $2g + n$ homology classes of curves. Of these, $2g$ curves will lie in distinct homology classes and $n$ will be null-homologous.
  \end{thm}


\section{Reduction to an adapted homology basis} \label{sec:reductionTObasis}
If we can specifically identify the $M$-generators using the procedure of section \ref{sec:Mremove} to remove them from an adapted generating set, we will have an {\sl adapted homology basis}.

\begin{defn} \label{def:adaptedHomologyBasis} Adapted Homology Basis is any set of adapted generators for the first homology from which the images of the null-homologous $M$-generators have been removed.
It will consists of $2g$ elements.
\end{defn}


\begin{thm} \label{thm:MAB} {\rm (Adapted Bases Exist)}
For any  $(\Gamma_0, S_0,  G)$ with  $S_0$ the quotient of the unit disc by the action of $\Gamma_0$  there exists  an integral homology basis adapted to the action of $G$ on the quotient surface $S=S_0/G$.gg 
The numbers of the different types of elements in the basis are determined by the surface kernel map or the generating vector in a manner that can be made precise.

\end{thm}

\begin{remk}
The difficulty is, of course, in identifying these generators. This can be done if one starts with the Cayley diagram of the group and/or its spanning tree. It can also be done if the set of Schreier generators has a maximal power element (see section \ref{sec:easier}). For prime order actions this is not a problem \cite{matrix, Jalg, India}. Some examples of adapted bases can be found in \cite{Link} and a  short-cut for computing the matrix using coset representatives with the maximal power property   \cite{inprep}.
\end{remk}

\section{The Matrix of the Action on an Adapted Homology Basis}
\label{sec:matrix}




\begin{cor}
The matrix of the action of $G$ with respect to an adapted will have a collection of sub-blocks of the following type:

\begin{enumerate}
\item Square permutation matrices.
\item Super permutation matrices. These are square matrices with $1$'s along the super-diagonal, every entry in the last row $-1$ and all other entries $0$.
    \item At most $(n-1)$ rows of $M$-type.

\end{enumerate}

\end{cor}

\section{Automorphism Lists} \label{sec:autos}

We use some of the automorphisms in the tables below. The initial list, due to Harvey \cite{Harvey}, has been expanded by Broughton-Wootton giving the tables below (page 6 of \cite{BW}):

$\begin{array}{|c|c|c|c|c|}
\hline
\; & \alpha_i & \beta_i & \alpha_j & \beta_j \\
 \hline \hline
U_i &  \alpha_i & \beta_i \cdot \alpha_i & \alpha_{i+1} & \beta_{i+1} \\
 \hline
 B_i & \alpha_i \cdot \beta_i & \beta_i & \alpha_{i+1} & \beta_{i+1}
  \\
  \hline
  R_i & \alpha_i \cdot \beta_i \cdot \alpha_i^{-1}   & \alpha_i^{-1} & \alpha_{i+1} & \beta_{i+1}\\
  \hline
  \sigma_i & \delta \cdot \alpha_{i+1} \cdot \delta^{-1} & \delta \cdot \beta_{i+1}  \cdot \delta^{-1} & \alpha_i & \beta_i \\
 \hline Z_i  & \alpha_i \cdot \alpha_{i+1} & \alpha_{i+1}^{-1} \cdot \beta_i \cdot \alpha_{i+1}  &  \epsilon \cdot \alpha_{i+1}  & \beta_{i+1}^{-1} \cdot \beta_i \cdot \epsilon \\
\hline \hline
\end{array}$

Here $\delta= [\alpha_i, \beta_i]$ and $\epsilon = [\alpha_{i+1}^{-1}, \beta_i]$

 $\begin{array}{|c|c|c|c|c|}
\hline
\; & \alpha_i & \beta_i & \gamma_j & Note \\
 \hline \hline
 B_{i,j} &  \alpha_i \cdot u \cdot \gamma_j \cdot u^{-1} & \beta_i & v  \cdot \gamma_j \cdot v^{-1} & u = \beta_i \alpha_i^{-1} \cdot \beta_i^{-1} \cdot w_2 \\
 \hline \hline
 U_{i,j} & \alpha_i & \beta_i \cdot x \cdot \gamma_j \cdot x^{-1}  & y \cdot \gamma_j \cdot y^{-1} & \mbox{with $x$ and $y$ as below} \\

  \hline \hline
  \end{array}$


  Here  $x =\alpha_i^{-1} \cdot \beta_i^{-1} \cdot w_2$ and $ y = x^{-1} \cdot  \alpha_i^{-1} x$
  where $w_2 = \Pi_{k+1}^{g_0} [\alpha_i, \beta_i] \Pi_{k=1}^{r-1} \gamma_k$.





\section{Acknowledgements}
This manuscript was completed over a number of years including times when I was a visitor at and/or supported by ICERM, CUNY Graduate Center, Princeton University. The term {\sl adapted basis} was coined in my PhD thesis. After I wrote the first draft of my thesis, Lipman Bers asked me to explain to him what I was describing about certain subgroups of the Mapping-class group (a.k.a the Tecum{\"u}ller{\"{u}}iller Modular group, the modular group) and their action. He said, "it needs a definition", call it having an {\sl adapted basis}. Others have since used this definition and the words he helped me coin. I also thank M. Condor for a helpful conversations, especially during a conference in Linkoping. 

 \bibliographystyle{amsplain}

\begin{thebibliography}{99}
\bibitem{Beard} Beardon, A. \textit{The Geometry of Discrete Groups} Springer-Verlag, Graduate Texts in Math \#91 (1983).

\bibitem{Rubi}  Behna, Antonio, Rodríguez, Rubi and Rojas, A.  \textit{  Adapted hyperbolic polygons and symplectic representations for group
actions on Riemann surfaces}   Journal of Pure and Applied Algebra 217 (2013) 409–426

\bibitem{Boughton1} Broughton, S. Allen \textit{ Simple Group Actions on hyperbolci Riemann surfaces of least area} Pacific J. Math 158 (1993), no. 1, 23-48.

\bibitem{BW} Broughton, S. Allen and-Wootton, A. \textit{ Finite abelian subgroups of the mapping class
group} Algebr.  Geom Topol \#7 (2007) 1651-1697.



\bibitem{MC1} Conder, Marston \textit{
Experimental Algebra}
Math. Chronicle
20
(1991), 1-11

\bibitem{MC2}
Conder, Marston D. E.
and
Kulkarni, Ravi S.
\textit{
Infinite families of automorphism groups of Riemann surfaces} in
Discrete groups and geometry (Birmingham, 1991),
London Math Soc Lecture Note Ser {\bf 173} Cambridge Univ Press
Cambridge
(1992).
\bibitem{GGH} Baginski, Czeslaw,
 Gromadzki, Grzegorz,
and  Hidalgo,Ruben A. \textit{ On purely non-free finite
actions of abelian groups on compact surfaces}, in press.
\bibitem{thesis}  Gilman, J. \textit{ Relative Modular Groups in Teichmüller Spaces}, Thesis, Colombia University (1971).

    \bibitem{JG1}  Gilman, J.
     \textit{ Compact Riemann Surfaces with Conformal Involutions},   Proceedings  Amer Math Soc {\textbf 37} (1973), 105-107.
\bibitem{Moduli} Gilman, J.
   \textit{ On the Moduli of Compact Riemann Surfaces with a Finite Number of Punctures} in {\sl Discontinuous Groups and Riemann Surfaces}, Annals of Math Studies {\textbf 79} (1974), 181-205.
\bibitem{matrix} Gilman, J. \textit{A Matrix Representation for Automorphisms of Riemann Surfaces},
Linear Algebra and its Applications {\textbf 17} (1977), 139-147.
\bibitem{GP} Gilman, J. \textit{Intersection Matrices for Adapted Bases} (with David Patterson) in {\sl Riemann Surfaces and Related Topics},  Annals of Math.
Studies {\textbf 97} (1981), 149-166.
  \bibitem{Jalg} Gilman, J. \textit{Canonical Symplectic Representations for Prime Order Conjugacy Classes of the Mapping Class Group},  Journal of Algebra, {\textbf 318} (2007), 430-455.
  \bibitem{India} Gilman, J. \textit{ Prime Order Automorphisms of Riemann surfaces}, in {\sl Proceedings of the International Workshop on Teichm{\"u}ller Theory and Moduli
Problems, HRI}  Lecture Notes Series, Ramanujan Mathematical Society, {\textbf 10}  (2009), 229-246.
\bibitem{Link}  Gilman, J. \textit{Computing Adapted Bases for Conformal Automorphism Groups of Riemann Surfaces}
Proc. Linkoping 2013 Conference, AMS Conn. Math., (2014) 137-153.
\bibitem{InPrep}  Gilman, J. \textit{The Maximal Power Property for Schreier Coset Representatives}, in preparation.
\bibitem{GP}  Gilman, J. and Patterson, D. \textit{Intersection Matrices for Adapted Bases} (with David Patterson) in {\sl Riemann Surfaces and Related Topics},  Annals of Math.
Studies {\textbf 97} (1981), 149-166.



























\bibitem{GabinoJ1} Gonzalez-Diez, Gabino and Harvey, W.J. \textit{  On families of algebraic curves with automorphisms} AMS Con Math 331  (2002), 262-237.
\bibitem{GabinoJ} Gonzalez-Diez, Gabino and Harvey, W.J. \textit{  Moduli of Riemann Surfaces with Symmetry} in Discrete groups and Geometry,   LMS Lecture Note Series no. 173 (1991), 75-93.

\bibitem{Gonz} Gonzalez-Aguilera, V and Rodriguez, R.E. \textit{ On principally polarized abelian varieties and Riemann surfaces associated to prisms and pyramids} in Lipa's legacy,
Contemp. Math. 211 AMS (1997), 269-284.

\bibitem{Harvey1} Harvey, William J. \textit{  Cyclic groups of automorphisms of a compact Riemann surface,} Quart. J. Math. Oxford Ser. 17 (1966) 86–97.
\bibitem{Harvey2} Harvey, W. J. \textit{  On branch loci in Teichm\"uller space,} Trans. Amer. Math. Soc. 153 (1971) 387–399.q


\bibitem{MKS}  Magnus, Karass , and Solitar \textit{Combinatorial Group Theory}, Wiley, (1966).
\bibitem{Maskit} Maskit, B. \textit{Kleinian groups}, Springer-Verlag Grundleheren der mathematischen Wissenschaften \# 287 (1987)
\bibitem{RR} Riera, G. and Rodriguez, R. \textit{ Riemann Surfaces and abelian varieties with an automorphism of prime order} Duke Math J. {\bf 69} (1993), 1990217.
    \bibitem{Springer} Springer, G. \textit{ An Introductions to Riemann Surfaces} Chelsea AMS, (2002) and  (1957). 
\end{thebibliography}

\part{Appendices}
\appendix

\section{Summary of Symbols and Conventions} \label{sec:symbols}

$\approx$ freely equal

$\approx^h$ homologous

${\overline{v}}$ the right coset representative of $v$.

$S_{K,v} = Kv{\overline{Kv}}^{-1}$

$g_v$ the image in $G$ of $\phi(v)$

When write $S_{W,z}$ without a bar over the $W$, it is understood that the first subscript is the coset representative of $W$.

Since tails and prefixes cancel, we write $Xg(X)g^2(X)... $ etc and interchange it with $S_{K,X}$ when $g = \phi(K)$ without further comment.



\section{Easier Computation with a Maximal Power Element} \label{sec:easier}
We begin with a lemma.
\begin{lem}\label{lem:m-1}
If $v$ is an order $m$ generator of $\Gamma_0$, then $S_{{\overline{x^{m-1}}},x}$ is an $M$ generator if and only if ${\overline{x^m}}= x^m$. In that case,  $S_{{\overline{Kx^q}},x}$ is an $M$-generator for each $q =1,...,m-1$ for each coset representative $K$,.
\end{lem}

\begin{proof}
We consider $\tau(x^m) =1$ where $x$ is an elliptic  generator of order $m$ and the relation
$S_{1,x}S_{{\overline{x^2}},x} \cdots S_{ {\overline{x^{m-2}}}, x}S_{{\overline{x^{m-1}}},x}
=1$

We will solve for
$S_{{\overline{x^{m-1}}},x}$ and eliminate it as a generator. We need to check whether
${\overline{x^{m-1}}}x =
{\overline{{\overline{x^{m-1}}}x}^{-1}} =1$, that is,  whether  ${\overline{x^{m-1}}}x =  {\overline{{\overline{x^{m-1}}x}^{-1}}}$
But the right hand side is ${\overline{x^m}} =1$.

This is ${\overline{x^{m-1}}} = x^{-1}$. But we write ${\overline{x^{m-1}}} = x^{m-1}$. Whence because these are Schreier generators ${\overline{x^q}}=x^q$ for all $q = 1,...,m-1$.

We also calculate in this case that $S_{Kx^{m-1},x}$ is an $M$-generator if $S_{x^{m-1},x}$ is one.

Thus if it is an $M$-generator we remove ${\frac{n}{m}}m$ images.

Thus we may assume that either one $x_j$ gives an $M$-generator or none do.
\end{proof}

\begin{cor}
If  $\Gamma_0$ contains one elliptic $x_i$ for which $S_{{\overline{x_i^{m-1}}},x_i} = 1$, then it contains exactly one such $x_i$ and in this case all  of the $n$ elements $S_{{\overline{Kx_i^q}},x_i}$ are $M$-generators and  thus are eliminated from the set of Schreier generators.
\end{cor}

We can also apply this to a hyperbolic generator where $m$ is the order if $\phi(v)$.

\begin{defn} Elements of this type are termed {\sl maximal power elements}.

\end{defn}
\begin{cor}
If  $\Gamma_0$ contains a hyperbolic generator $v$ for which $S_{{\overline{v^{m-1}}},v} = 1$, then it contains exactly one such $v$ and in this case all  of the $n$ elements $S_{{\overline{Kv^q}},v}$ are $M$-generators and  thus are eliminated from the set of Schreier generators.
\end{cor}

It follows from the above, that if the set Schreier representatives contains a maximal power element, then the adapted basis will not have rows of type $M$ because all of the $M$-words will be of the form  $S_{K,v}$ or $S_{K,x}$ depending upon whether the elements is a hyperbolic $v$ or an elliptic $x$. 

\section{Worked Example}\label{sec:ex}


During the  September $25-29, 2017$ workshop at Banff, Grzegorz Gromadzki asked about an adapted homology basis for the group $G= {\mathbb{Z}}_5 \times {\mathbb{Z}}_5$  acting on a surface $S$ of genus $g$ with quotient surface $S_0$ of genus $g=0$.

We compute an adapted homology basis reflecting for the action of the group  $G$ where the surface $S_0$ is  uniformized by $\Gamma_0$ with presentation
\begin{equation}
\Gamma_0 =
\langle x_1,x_2,x_3,x_4,x_5,x_5,x_6| x_j^5; j=1,...,6; x_1x_2x_3x_4x_5x_6 \rangle.
\end{equation}
If $U$ denotes the unit disc or the upper half plane model for hyperbolic two-space, we have $S_0= U\Gamma_0$. Let $\phi: \Gamma_0 \rightarrow G$. Set $\Gamma = \mbox{Ker} \phi$, the kernel of $\phi$. We choose generators for $G$ given by Grzegorz Gromadzki listed in a section below \ref{sec:gens} and choose the $25$ coset representatives also listed there. Note that the genus of $g$ of $S$ is $36$ as calculated by the Riemann Hurwitz relation so that the matrix of the action of $G$ on an integral homology basis for $S$ will be a $2g \times 2g $ matrix, that is,  a $72 \times 72$ matrix.

\subsection{Notation  and the Schreier-Reidemeister Theory and Theorem}
We recall the notation and the statement of the Reidemeister Theory and the Reidemeister Theorem as applied to this situation using \cite{MKS}.

Namely, for any element  $v \in \Gamma_0$ we denote its coset representative by ${\overline{v}}$ and a generator for $\Gamma_0$ by $S_{K,x_j}$ where $K$ runs over the complete set of $25$ coset representatives
$\{1, x_j^i, j=1,...6, i=1,...,4\}$ and $S_{K,v}= KV{\overline{Kv}}^{-1}$.

Here we have  chosen $25$  right coset representatives as listed below (section \ref{sec:gens}) and we let $K$ denote an arbitrary one of the right coset representatives.

We remind the reader that a Schreier-Reidemeister rewriting process $\tau$ writes a word in the generators of $\Gamma_0$ that lies in $\Gamma$ in terms of the generators $S_{K,v}$ for $\Gamma$ where $K$ runs over a complete set of Schreier  minimal right coset representatives. 


The rewriting process is defined as follows: Let $X= a_1^{\epsilon _i} \dots a_u^{\epsilon_u}$, where each $\epsilon_i$ is either $+1$ or $-1$ and $a_i$ is a generator of $\Gamma_0$.
Then $$\tau(X) = \Pi_{i=1}^r S_{V_i,a_i}^{\epsilon_i},$$
 where $V_i$ depends upon $\epsilon_i$. Namely, if $\epsilon_i = +1$, then $V_i = {\overline{a_1^{\epsilon_1} \dots a_{i-1}^{\epsilon_{i-1}}} }$ and if
$\epsilon_i = -1$, then $V_i = {\overline{a_1^{\epsilon_1} \dots a_{i-1}^{\epsilon_{i-1}} a_i^{-1}}}$.

We apply Theorem 2.9 page 94 of \cite{MKS} to our situation to obtain

 \begin{thm}
  We have generators $S_{K,x_j}, j = 1, ..6 $ with relations
  $$\tau(KRK^{-1}) \mbox{and} \tau(Kx_j^5K^{-1}), j= 1,...,6 \mbox{where} R= x_1x_2x_3x_4x_5x_6.$$
   Here $K$ runs over the full set of coset representatives.

 Further $S_{M,v} \approx 1$ if ${\overline{Mv}} = Mv$ and such elements are termed {\sl $M$-generators}.
 \end{thm}

 $G$ acts on $\Gamma_0$ by conjugation and we denote the image of $\phi(v)$ by $g_v$ for any word $v \in \Gamma_0$.

 We think of the images of the $\tau(Kx^5K^{-1})$ as being broken up into $5$ non-conjugate word products
  each of length $5$ and we refer to theses as {\sl lines} (see section \ref{sec:lines}).
\subsection{Elimination of generators and relations}

We eliminate generators and relations using Tietze transformations of types T2 and T4 (page 50 of \cite{MKS})  that correspond to gluing  or to solving for a single generator and replacing or eliminating.

We follow the steps listed below with some notation and details as described in the sections referenced.  

\begin{description}
  \item[Step 1] We consider the $\tau(Kx_6^5K^{-1})$ which yield $6$ distinct relations each involving $5$ generators. We remove these $30$ generators and use these relations to glue the $25$ $\tau(KRK^{-1})$ relation into $5$ relations.

  \item[Step 2] We consider the $5$ non-conjugate lines $\tau(Kx_5^5K^-1)$ and use one of these lines to glue the five remaining relations to one single relation. For simplicity assume this is line \#5. (See sections \ref{sec:lines} and \ref{sec:x5lines}.)
        \item[Step 3] We remove the $24$ non-identity Schreier generators (section \ref{sec:gens}). In the situation for Schreier generators involving $x_6$ this creates some complications addressed in the next step.
  \item[Step 4] We consider the $\tau(xRx^{-1})$ and find the values of $x$ for which $S_{1,x_6}, S_{x_6,x_6},S_{x_6^2,x_6},S_{x_6^3,x_6}$ occurs.  Since all of these are $\approx 1$, we can solve for the $S_{{\overline{xx_1x_2x_3x_4}},x_5}$ occur that occur for each of the four. Call these four $V_1,V_2,V_3,V_4$. We can solve these in terms of the four other words in the line each involving four Schreier generators with $x_1,x_2,x_3,x_4$ (section \ref{sec:x5lines}).
  \item[Step 5] For each of the five $x_j$ lines, where $j = 1,2,3,4$, we can eliminate the first line $\tau(x_j^5)$ as its first elements are $M$-generators making the last elements $\approx 1$. This eliminates $4 \times 5=20$ Schreier generators  (section \ref{sec:lineone}).
  \item[Step 6] For each of the four remaining such lines we solve for the last element keeping the first four and calculate the action of $G$ (or $g_x$ as described below (sections \ref{sec:gx} and \ref{sec:lines2-5}). We obtain for each $x_j$ and its images $4$ matrices that are $4 \times 4$ of the form: The matrix has $16$ $4\times 4$ sub-blocks.
                                         of the form $M_4= \left(
                                                       \begin{array}{cccc}
                                                         0 & 1 & 0 & 0 \\
                                                         0 & 0 & 1 & 0 \\
                                                         0 & 0 & 0 & 1 \\
                                                         -1 & -1 & -1 & -1 \\
                                                       \end{array}
                                                     \right)$
  \item[Step 7] We have the full $72 \times 72$ matrix with a $64 \times 64$ submatrix broken into $16$ $4 \times 4$ blocks or $M_{16}= \left(
                                                       \begin{array}{cccc}
                                                         M_4 & 0 & 0 & 0 \\
                                                         0 & M_4 & 0 & 0 \\
                                                         0 & 0 & M_4 & 0 \\
                                                         0 & 0 & 0 & M_4 \\
                                                       \end{array}
                                                     \right)$
  \item[8] Finally, we look at the remaining $3$ lines from $x_jx_5^5x_j{-1}$ where $j = 2,  3,4$
  and obtain as in section \ref{sec:x5lines} (or section \ref{sec:alt5})  a $2\times 2$ block and a $72 \times 1$ for each of these three lines completing the $72 \times 72 $ matrix.
\end{description}


\subsection{Lines} \label{sec:lines}

We refer to the non-conjugate {\sl lines} of $x_j$, $j = 1,...,6$
as
\begin{description}
   \item[line \#1] $\tau(K_1T K_1^{-1} )= \Box \Box \Box \Box \Box$
   \item[line \#2] $\tau(K_2TK_2^{-1}) =\Box \Box \Box \Box \Box$
   \item[line \#3] $\tau(K_3TK_3^{-1})=\Box \Box \Box \Box \Box$
   \item[line \#4] $\tau(K_4TK_4^{-1})=\Box \Box \Box \Box \Box$
   \item[line \#5] $\tau(K_5TK_5^{-1})=\Box \Box \Box \Box \Box$
   \item[line \#6] $\tau(K_6TK_6^{-1})=\Box \Box \Box \Box \Box$
 \end{description}

where $K_1=id$ and $\{K_i| i = 1, 2,...,6 \}$ and $T=x^5$ as these run over a complete set of coset representatives that gives over the set of $6$ lines a complete set of $25$ Schreier generators for $\Gamma$. However for $T=x_j^5$ for each $j$ in the line $\#j$
 is the same up to conjugacy as line $\#1$.

For example, in computing
$\tau(x_1x_2^5x_1^{-1})= $
$S_{1,x_1}S_{x_1,x_2}S_{{\overline{x_1x_2}},x_2}S_{{\overline{x_1x_2,x_2}},x_2}S_{{\overline{x_1x_2^2}},x_2}
S_{{\overline{x_1x_2^3}},x_2}S_{{\overline{x_1x_2^4}},x_2}S_{1,x_1}^{-1} \approx $

$ S_{x_1,x_2}S_{{\overline{x_1x_2}},x_2}S_{{\overline{x_1x_2,x_2}},x_2}S_{{\overline{x_1x_2^2}},x_2}
S_{{\overline{x_1x_2^3}},x_2}S_{{\overline{x_1x_2^4}},x_2}$, we need to compute the coset representatives ${\overline{x_1x_2^t}}$. Here the $S_{1,x_1}$ terms have been omitted because they are freely equal to $1$.

For the sake of brevity, we note that a typical {\sl line} looks like
$$ \Box \Box \Box \Box \Box =1$$ and speak of the first, second ... , last element of the line when the subscripts are understood.

\subsection{The $24$ non-identity  Generators} \label{sec:gens}
 Greg  asked for the following generators:
 \vskip .1in
 $(3,0), (4,4),(1, 2), (1,3), (1,4), (0,2)$

 $(3,0)^2 = (6,0) = (1,0) $

 $(3,0)^3 = (3,0)(1,0) = (4,0)$

  $(4.4)^2= (3,3)$

  $(4,4)^3= (1,1)$

  $(1,2)^2 =(2,4)$

  $(1,2)^3= (2,0)$


  $(1,3)^2 = (2,1)$

  $(1,3)^3= (3,4)$

  $(1,4)^2 = (2,3)$

  $(1,4)^3 = (3,1)$

  $(0,2)^2 = (0,4)$

  $(0,2)^3 = (0,1)$
\vskip .05in
Thus we   can take $1$,$x_i$, $x_i^2, x_i^3 \mbox{and}   x_i^4$ $ i= 1,..., 6$ as coset reps. These are $25$ and all will give $M$-generators.
We have

We can compute the images of the powers of the $x^j$ under $\phi$.
\vskip .1in
$\phi(x_1) = (3,0)$
$\phi(x_1^2) = (1,0)$
$\phi(x_1^3) = (4,0)$
$\phi(x_1^4) = (2,0)$
\vskip .1in

$\phi(x_2) = (4,4)$
$\phi(x_2^2) = (3,3)$
$\phi(x_2^3) = (2,2)$
$\phi(x_2^4) = (1,1)$
\vskip .1in
$\phi(x_3) = (1,2)$
$\phi(x_3^2) = (2,4)$
$\phi(x_3^3) = (3,1)$
$\phi(x_3^4) = (4,3)$
\vskip .1in
$\phi(x_4) = (1,3)$
$\phi(x_4^2) = (2.1)$
$\phi(x_4^3) = (3,4
)$
$\phi(x_4^4) = (4,2)$
\vskip .1in
$\phi(x_5) = (1,4)$
$\phi(x_5^2) = (2, 3)$
$\phi(x_5^3) = (3, 2)$
$\phi(x_5^4) = (4, 1)$
\vskip .1in
$\phi(x_6) = (0,2)$
$\phi(x_6^2) = (0, 4)$
$\phi(x_6^3) = (0,1)$
$\phi(x_6^4) = (0,3)$

\vskip .1in

\subsection{Images of $g_x$ and its powers} \label{sec:gx}
We compute the action of $G$ on $\gamma$ noting that $g_x(W)= \tau(xWx^{-1})$ where $W$ is any word in the Schreier generators if $\phi(x) = g_x$. Need to fix or add


Usually if we compute $g_W(S_{1,v}) =\tau(w_1 \cdots w_tS_{1,v}(w_1 \cdots w_t)^{-1}$ where $W= w_1 \cdot w_t$ is a minimal Schreier generator  we obtain $S_{W,v}$ multiplied by the product of words denoted by $T_{W,v}^{-1}$   so that $g_W(S_{1,v}) = S_{W,v}T_{W,v}^{-1}$.
We term $T_{W,v}$  the {\sl Tail} of $g_W$.
However, in this case because all of our Schreier representatives are $M$-generators we have a simpler situation.

We need to compute $g_x^t(S_{x^q,x})$ where $x$ is an $x_i, i = 1,...,5$ and $q= 1, ..., 4$.

$\tau(x^tS_{x^q,x}x^{-t}) = $
$$S_{1,x}S_{x,x} \cdots S_{{x^{q-1}},x} S_{{x^q},x}S_{x^q,x}^{-1}(S_{1,x}S_{x,x} \cdots S_{{x^{q-1}},x})^{-1}.$$

But because $x^t$ is an $M$-generator for each $t=1,2,3,4$, the first and last $q$ elements are freely equal to $1$.


If $\gamma=S_{1,x}$ and ${\overline{x}} = x$
$$g_x(\gamma) = \tau (xx{\overline{x}}^{-1}x^{-1}) = S_{1,x}S_{{\overline{x}},x}
S_{{\overline{x}},{\overline{x}}x}S_{1,x}^{-1} = S_{{\overline{x}},x}
S_{{\overline{x}},{\overline{x}}}$$

We also compute that in $\tau(xxxxxxx^{-1})= g_x(\tau(x^5))$
 $S_{x,x}$ has a prefix which we denote by $P_{x,x}$ and that $P_{x,x}= T_{1,x}^{-1}$ so that these will cancel and if $\gamma= S_{1,x}$, then $\tau(x^n)=1$ can be written as $$\gamma g_x(\gamma) g_x^2(\gamma)g_x^3(\gamma)g_x^4(\gamma)=1$$
and a similarly for $\tau(R)$ so that if
 $\gamma_1 =S_{1,x_1}$,  $\gamma_2 =S_{1,x_2}$,  $\gamma_3 =S_{1,x_3}$, $\gamma_4 = S_{1,x_4}$  $\gamma_5 =S_{1,x_5}$ and  $\gamma_6 =S_{1,x_6}$, we have
$$\tau(R)  = \gamma_1 \cdot [ (g_{x_1x_2})(\gamma_2)]\cdot [g_{x_1 \cdot x_2 \cdot x_3}(\gamma_3)]\cdot [g_{x_1x_2x_3x_4}(\gamma_4)] \cdot [g_{x_1x_2x_3x_4x_5}(\gamma_5)]  \cdot [g_{x_1x_2x_3x_4x_5x_6}(\gamma_6)]. $$
Here we used
$\tau(R) = S_{1,x_1}S{x_1,x_2}S_{x_3^2,x_3}S_{x_3^3,x_4}S_{x_2,x_5}S_{x_6^4,x_6}   $

Below we carefully compute $\tau(yx^5y^{-1}) =S_{y,x}S_{{\overline{yx}},x}S_{{\overline{yx^2}},x}S_{{\overline{yx_3}},x}S_{{\overline{yx^4}},x} =1$
where $x = x_1,x_2,x_3,x_4,x_5 \; \mbox{or} \; x_6$.


\subsection{Elimination} \label{sec:elimination}

In this section we carry out very carefully some of  eliminations referred to earlier.
\subsubsection{Elimination of line \#1} \label{sec:lineone}
For each $x_j$ line $\# j$ is the same as line $\# 1$. Thus we adjust the numbering for each $x_j$ so that there are $5$ lines eliminating what was previously termed line $\#j$.

For each $x_j, j=1,...,5$, line \#1 gives $\tau(x^5)= \Box \Box \Box \Box \Box$
But the first four boxes represent $M$ generators $\approx 1$. Thus the last box does not contribute a generator.

   \subsubsection{Elimination along lines $\#2-5$ for $x_j, j = 1,2,3,4$} \label{sec:lines2-5}

   For each $x_j, j=1,...,4$ and for each coset representative $y$ we have $\tau(yRy^{-1})= \Box \Box \Box \Box \Box$. We use one image for each of the five non-conjugate lines.
   Thus we can solve for generator in the last box and replace it by the product of the inverses of the others.

   Thus each lines  \#(2-5), we have $4$ generators. This is a total of $16$ generators.

   That is $16 \times 4 = 64$ generators.

   \subsubsection{Elimination along the $x_5$ lines} \label{sec:x5lines}
   We for $x_5$ eliminate line \#5 for gluing and we  eliminate line \#1.

   This leaves $3$ lines each of which contributes $4$ generators or a total of $12$ generators.

   That is $76$ generators. To compensate note that
 when we computed $\tau(Kx_1x_2x_3x_4x_5x_6K^{-1})$, we obtained four  of the form $S_{x_6^i,x_6} \approx 1, \;\; i=1,2,3,4$.
   For these, we obtain a relation that $S_{**,x_5}^{-1}$ can be written as the product of words in
   the $S_{***,x_t}$ where $t=1,2,3, \mbox{and } 4$. Here The $**$ and the $***$ denote the appropriate first subscript in Schreier generators that occur. This can be computed.

   Thus in the $3$ lines each with $4$ generators, we can eliminate the last one involving an $S_{***,x_5}$
   and in computing the matrix of the action instead of a $4 \times 4$ matrix, we obtain a $3 \times 3$ matrix of the form with a $3 \times 3$ matrix $\left(
                                            \begin{array}{ccc}
                                              0 & 1 & 0\\
                                              0 & 0 & 1\\
 0 & 0 & 0 \\
                                             \end{array}
                                          \right)$
   but the last row is part of the row of length $72$ where there are other entries which can be calculated. 
   These will be $0$, $1$ or $-1$.

That is, we omit a further $4$ generators and have $72$ remaining.

\subsubsection{Alternate phrasing} \label{sec:alt5}

We have $\tau(xx_5^5x^{-1})= S_{x,x_5} S_{{\overline{xx_5}},x_5}
S_{{\overline{xx_5^2}},x_5}S_{{\overline{xx_5^3}},x_5}S_{{\overline{xx_5^4}},x_5}$
So we eliminate
$S_{{\overline{xx_5^4}},x_5}$ and replace it by the product of inverses.
We also have $\tau(yx_1x_2x_3x_4x_5x_6y^{-1}) =1$
But as $y$ varies over $x_1,x_2,x_3,x_4,x_5,x_6$ we obtain $4$ instances where the
$x_6$ term is $S_{1,x_6}$, or
$S_{x_6,x_6}$, $S_{x_6^2,x_6}$, $S_{x_6^3,x_6}$

All of which are $\approx 1$. Thus the term $S_{K,x_6}$ drops out and we can then solve for the $S_{{\overline{Kx_6^{-1}}},x_5}$ in terms that only involve $x_1,x_2, x_3 \mbox{or} x_4$
as the subscript in the Schreier generators.

\subsubsection{Further notation} \label{secx:x5x6}

We note
$\tau(xx_1x_2x_3x_4x_5x_6x^{-1}) =$
$$S_{x,x_1} S_{{\overline{xx_1}},x_2} S_{{\overline{xx_1x_2}},x_3}S_{{\overline{xx_1x_2x_3}},x_4}S_{{\overline{xx_1x_2x_3x_4}},x_5} = S
_{{\overline{xx_1x+2x_3x_4x_5}},x_6}^{-1}.$$

Solve for $x$ such that the right hand side is one of $S_{1,x_6}$, $S_{x_6,x_6}$, $S_{x_6^2,x_6}$, or
$S_{X_6^3,x_6}$ each of which is $\approx 1$.  There are four such $x$, call them $Y_1,Y_2,Y_3,Y_4$

The solve for the corresponding $$S_{{\overline{xx_1x_2x_3x_4}},x_5}= (S_{x,x_1} S_{{\overline{xx_1}},x_2}S_{{\overline{xx_1x_2}},x_3}S_{{\overline{xx_1x_2x_3}},x_4})^{-1}
S_{{\overline{xx_1x_2x_3x_4}},x_5}$$ where $x = Y_1,Y_2,Y_3,Y_4$.
Then looks at the lines $\tau(x_ix_5^4x_i^{-1})= A_iB_iC_iD_iE_i$.
We have $g_x(A_i)=B_i, g_x(B_i)=C_i, g_x(C_i)=D_i,g_x(D_i)=E_i, g_x(E_i) =A_i$.
We have solutions for $E_i$ in terms of $A_i,B_i,C_i,D_i$ but also solutions for $E_i$ in terms of $(S{{\overline{*}},x_1}S_{{\overline{*x_1}},x_2}
S_{{\overline{*x_1x_2}},x_3}S_{{\overline{*x_1x_2x_3}},x_4})^{-1}$
Thus we eliminate the $E_i$ and the $g(D_i)$ to obtain for each of the $3$ lines, the elements
$A_i,B_i,C_i$ giving a $2 \times 3$ matrix of the form $\left(
                                           \begin{array}{ccc}
                                             0& 1 & 0\\
                                             0 & 0 & 1 \\
                                           \end{array}
                                         \right)$
                                         and a
                                         $72 \times 1$
                                                                                  matrix with $0$'s in the $A_i,B_iC_i$ slots but appropriate $-1$'s in other generator's slots.

                                         We call this line in the matrix an $Y_iL_5$ line.

                                         \subsection{The $72 \times 72$ Matrix} \label{sec:matrix}

                                         The matrix has $16$ $4\times 4$ sub-blocks
                                         of the form $\left(
                                                       \begin{array}{cccc}
                                                         0 & 1 & 0 & 0 \\
                                                         0 & 0 & 1 & 0 \\
                                                         0 & 0 & 0 & 1 \\
                                                         -1 & -1 & -1 & -1 \\
                                                       \end{array}
                                                     \right)$

                                                     and four other $2 \times 2$ submatrix blocks along with
                                                     $3$ lines of the form $Y_iL_5$ $i=1,2,3$.

                                                     This gives a $ 72 \times 72$ matrix that gives the action of $G$ using where each entry is only $1,-1,\; \mbox{or} \;0$



\end{document}